\documentclass[11pt]{amsart}
\usepackage{verbatim}
\usepackage{amssymb, amsmath}
\usepackage{graphicx}
\usepackage{color}

\begin{document}
\newtheorem{thm}{Theorem}[section]
\newtheorem*{thm*}{Theorem}
\newtheorem{lem}[thm]{Lemma}
\newtheorem{prop}[thm]{Proposition}
\newtheorem{cor}[thm]{Corollary}
\newtheorem*{conj}{Conjecture}
\newtheorem{proj}[thm]{Project}

\theoremstyle{definition}
\newtheorem*{defn}{Definition}
\newtheorem*{remark}{Remark}
\newtheorem{exercise}{Exercise}
\newtheorem*{exercise*}{Exercise}

\numberwithin{equation}{section}

\newcommand{\rad}{\operatorname{rad}}

\def\xo{x_{1}}
\def\yo{x_{2}}
\def\xt{x_{3}}
\def\yt{x_{4}}

\def\no{n_{1}}
\def\nt{n_{2}}

\def\etao{\eta}
\def\etat{2\eta}
\def\vx{\mathbf{x}}

\global\long\def\xo{\mathrm{x_{1}}}%
\global\long\def\xt{\mathrm{x_{3}}}%
\global\long\def\yo{\mathrm{x}_{2}}%
\global\long\def\yt{\mathrm{x_{4}}}%

\global\long\def\xbo{\mathrm{\XBox_{1}}}%
\global\long\def\xbt{\mathrm{\XBox_{3}}}%
\global\long\def\ybo{\mathrm{\XBox}_{2}}%
\global\long\def\ybt{\mathrm{\XBox_{4}}}%

\newcommand{\Z}{{\mathbb Z}} 
\newcommand{\Q}{{\mathbb Q}}
\newcommand{\R}{{\mathbb R}}
\newcommand{\C}{{\mathbb C}}
\newcommand{\N}{{\mathbb N}}
\newcommand{\FF}{{\mathbb F}}
\newcommand{\fq}{\mathbb{F}_q}
\newcommand{\rmk}[1]{\footnote{{\bf Comment:} #1}}

\renewcommand{\mod}{\;\operatorname{mod}}
\newcommand{\ord}{\operatorname{ord}}
\newcommand{\TT}{\mathbb{T}}
\renewcommand{\i}{{\mathrm{i}}}
\renewcommand{\d}{{\mathrm{d}}}
\renewcommand{\^}{\widehat}
\newcommand{\HH}{\mathbb H}
\newcommand{\Vol}{\operatorname{vol}}
\newcommand{\area}{\operatorname{area}}
\newcommand{\tr}{\operatorname{tr}}
\newcommand{\norm}{\mathcal N} 
\newcommand{\intinf}{\int_{-\infty}^\infty}
\newcommand{\ave}[1]{\left\langle#1\right\rangle} 
\newcommand{\Var}{\operatorname{Var}}
\newcommand{\Prob}{\operatorname{Prob}}
\newcommand{\sym}{\operatorname{Sym}}
\newcommand{\disc}{\operatorname{disc}}
\newcommand{\CA}{{\mathcal C}_A}
\newcommand{\cond}{\operatorname{cond}} 
\newcommand{\lcm}{\operatorname{lcm}}
\newcommand{\Kl}{\operatorname{Kl}} 
\newcommand{\leg}[2]{\left( \frac{#1}{#2} \right)}  
\newcommand{\Li}{\operatorname{Li}}

\newcommand{\sumstar}{\sideset \and^{*} \to \sum}

\newcommand{\LL}{\mathcal L} 
\newcommand{\sumf}{\sum^\flat}
\newcommand{\Hgev}{\mathcal H_{2g+2,q}}
\newcommand{\USp}{\operatorname{USp}}
\newcommand{\conv}{*}
\newcommand{\dist} {\operatorname{dist}}
\newcommand{\CF}{c_0} 
\newcommand{\kerp}{\mathcal K}

\newcommand{\Cov}{\operatorname{cov}}
\newcommand{\Sym}{\operatorname{Sym}}

\newcommand{\Ht}{\operatorname{Ht}}

\newcommand{\E}{\operatorname{E}} 
\newcommand{\sign}{\operatorname{sign}} 
\newcommand{\meas}{\operatorname{meas}} 

\newcommand{\divid}{d} 

\newcommand{\GL}{\operatorname{GL}}
\newcommand{\SL}{\operatorname{SL}}
\newcommand{\re}{\operatorname{Re}}
\newcommand{\im}{\operatorname{Im}}
\newcommand{\res}{\operatorname{Res}}
\newcommand{\Str}{\mathcal S^{\rm trunc}} 
 \newcommand{\length}{\operatorname{length}}
 
\newcounter{ToDo}
\newcommand{\todo}[1]{\stepcounter{ToDo}%
	({\color{blue}TODO~\arabic{ToDo}: {#1}})%
}

\title[Pair correlation of lacunary sequences]{The metric theory of the pair correlation function of real-valued lacunary sequences}
\author{Ze\'ev Rudnick and Niclas Technau}
\date{\today}
 \address{School of Mathematical Sciences, Tel Aviv University, Tel Aviv 69978, Israel}
 
\email{rudnick@tauex.tau.ac.il}
\email{niclast@mail.tau.ac.il}
\thanks{This result is part of a project that received funding from the European Research Council (ERC) under the European Union's  Horizon 2020 research and innovation programme  (Grant agreement No.    786758).}

\begin{abstract}
Let $\{ a(x) \}_{x=1}^{\infty}$ be a positive, real-valued, 
lacunary sequence.
This note shows that the pair correlation function 
of the fractional parts of the dilations $\alpha a(x)$
is Poissonian for Lebesgue almost every $\alpha\in \mathbb{R}$. 
By using harmonic analysis, our result 
--- irrespective of the choice of the real-valued sequence 
$\{ a(x) \}_{x=1}^{\infty}$ ---
can essentially be reduced to showing that
the number of solutions to the Diophantine inequality 
$$
\vert n_1 (a(x_1)-a(y_1))- n_2(a(x_2)-a(y_2)) \vert < 1 
$$
in integer six-tuples $(n_1,n_2,x_1,x_2,y_1,y_2)$ located in the box 
$[-N,N]^6$ with the ``excluded diagonals'', that is 
$$x_1\neq y_1, \quad x_2 \neq y_2, \quad (n_1,n_2)\neq (0,0),$$ 
is at most $N^{4-\delta}$ for some fixed $\delta>0$, 
for all sufficiently large $N$. 

\end{abstract}
\maketitle

\section{Introduction}
A sequence of points $\{\theta_n\}_{n=1}^\infty$  
is uniformly distributed modulo one if given any fixed interval  $I$ in the unit circle $\R/\Z$, 
 the proportion of fractional parts $\theta_n \bmod 1$ which lie in $I$ tends to the length of the interval $I$, that is 
 \begin{equation*}
 \#\{n\leq N: \theta_n\bmod 1 \in I\} \sim  \length(I) \cdot N, \quad N\to \infty .
 \end{equation*}
 
 We study the {\em pair correlation function}   $R_2$, defined for every fixed interval $I\subset \R$ by the property that
\[
\lim_{N\to \infty} \frac 1N \#\{ 1\leq m\neq n\leq N: |\theta_n-\theta_m| \in \frac 1N I \} = \int_I R_2(x)dx
 \]
assuming that the limit exists. For a random sequence of $N$ elements, that is $N$ uniform independent random variables in $[0,1)$ (the Poisson model), the limiting pair correlation function is  $R_2(x)\equiv 1$.

There are very few positive results on the pair correlation function available for specific sequences,  a notable exception being the fractional parts of $\sqrt{n}$ \cite{EMV}; 
a more tractable problem is to randomize (a ``metric'' theory, in the terminology of uniform distribution theory) by looking at 
random multiples\footnote{A different notion of randomizing has recently been investigated in \cite{AB}, which studies the pair correlation of the sequence $\alpha^n \bmod 1$ with $\alpha$ random.}
$\theta_n=\alpha a(n) \bmod 1$, for almost all $\alpha$. 
 There is a well-developed metric theory of the pair correlation function for  {\em integer} valued sequences $\{a(n)\}_{n=1}^\infty$, initiated in  \cite{RudnickSarnak}, where polynomial sequences such as $a(n)=n^d$, $d\geq 2$, are studied,  
 with several developments in the last few years, see e.g. \cite{ALL, BCC, CLZ, LT, LarcherStockinger, PS, RZ pc, Walker}.  
 In this note we study the case of real-valued lacunary sequences:  
Let $a(x)>0$ be a lacunary sequence of positive reals, that is there is some $C>1 $ so that for all integers $x\geq 1$, 
\[
a(x+1)\geq C a(x).
\] 
For instance, we can take $a(x) = e^x$. It is known  that for almost all $\alpha$, the sequence $\alpha a(x)\bmod 1$ is uniformly distributed mod one \cite[Chapter 1, Corollary 4.3]{KN}. 
Here and throughout this note ``almost all'' is meant
with respect to the Lebesgue measure on $\R$.

\begin{thm}\label{thm:lacunary} 
Assume that $\{a(x)\}_{x=1}^\infty$
is a lacunary sequence of positive reals. 
Then the pair correlation function of 
the sequence $\{\alpha a(x)\}_{x=1}^\infty$ 
is Poissonian for almost all $\alpha$. 
\end{thm}

When $a(x)$ takes {\em integer} values,   \cite{RZ pc}  showed that for almost all $\alpha$, the pair correlation function is Poissonian. 
The case of  pair correlation of sequences of {\em rationals} $x_n=a_n/b_n$ with $a_n$ integer-valued and lacunary and $b_n$ integer-valued and (roughly speaking) sufficiently small (e.g.  $a_n/b_n=2014^n/[\log\log n]$) was treated in \cite{CLZ}.   Here we treat any real-valued sequences.

We will reduce the problem to giving a bound for the number of lattice points satisfying a Diophantine inequality: 
Let 
$$M=N^{1+\varepsilon}, \quad K=N^\varepsilon
$$ 
 and assume that there is some $\delta>0$ so that for all $\varepsilon>0$ sufficiently small
\begin{equation}\label{cond expec} \tag{A}
\#\{1\leq n\leq M, 1\leq x\neq y\leq N: n|a(x)-a(y)|<K \} 
\ll N^{2-\delta}.
\end{equation}
Let $\mathcal S(N)$ be the set of integer six-tuples with 
\[
1\leq y_i\neq x_i\leq N, \quad 1\leq |n_i|\leq M, \quad (i=1,2),
\]
satisfying 
\[
| n_1(a(x_1)-a(y_1))- n_2(a(x_2)-a(y_2)) |< K.
\]
Assume that 
\begin{equation}\label{lacunary ineq}\tag{B}
\#\mathcal S(N) \ll  N^{4-\delta}.
\end{equation}

\begin{thm}\label{prop:Reduce pc to ineq}
Let $\{a(x)\}_{x=1}^\infty$ be a sequence of distinct 
positive reals. 
Assume that  \eqref{cond expec} and \eqref{lacunary ineq} hold 
for some $\delta>0$.    
Then the pair correlation function of $\alpha a(x)$ 
is Poissonian 
for Lebesgue almost all $\alpha\in \mathbb{R}$.
\end{thm}

In the case of integer-valued sequences, the almost sure convergence of the pair correlation function to the Poisson limit (metric Poisson pair correlation) follows \cite{RudnickSarnak, RZ pc} 
from  a similar bound for the equation 
$$
n_1 (a(x_1)-a(y_1))- n_2(a(x_2)-a(y_2))=0
$$ 
See  \cite{ALL, BCGW} for  a  streamlined criterion for metric Poisson pair correlation in terms of the {\em additive energy}  of the sequence. 

In \S~\ref{sec:lacunary}, we verify that that  \eqref{cond expec} and \eqref{lacunary ineq} hold for lacunary sequences, hence obtain Theorem~\ref{thm:lacunary}.

 \section{The pair correlation function}\label{sec:pc} 
 To study the pair correlation function, we use a smooth count cf. \cite{RudnickSarnak}: 
For $f\in C_c^\infty(\R)$ or $f$ being an indicator function
of a compact interval, set 
 \[
 F_N(x) = \sum_{j\in \Z} f\big(N(x+j)\big) 
 \]
 which is periodic and localized on scale $1/N$. 
 For a sequence $\{\theta_n\}_{n=1}^\infty\subset \R/\Z$, 
 we define its pair correlation function by
\begin{equation}
R_2(f,N)(\{\theta_n\}_{n=1}^\infty)= \frac{1}{N} 
\sum_{1\leq m\neq n\leq N} 
F_N(\theta_n-\theta_m).
\end{equation}  
In particular, for a fixed sequence $\{x_n\}_{n=1}^\infty$, 
we take $\theta_n=\alpha x_n\bmod 1$,  
and abbreviate the pair correlation function 
$R_2(f,N)(\{\theta_n\}_{n=1}^\infty)$, having fixed $f$,
by $R_2(f,N)(\alpha)=R_2(\alpha)$. 

It suffices to restrict $\alpha$ to lie in a fixed  finite interval  and to consider a smooth average: 
Let   $\rho\in C_c^\infty(\R)$, $\rho\geq 0$, be a smooth, compactly supported, non-negative weight function, normalized to give a probability density:   $ \int_{\R}\rho(\alpha)d\alpha=1$.  We define a smooth average 
 \begin{equation}\label{eq: smooth ave}
  \ave{X} = \int_{\R} X(\alpha)\rho(\alpha) \, \mathrm{d}\alpha . 
  \end{equation}  

 \subsection{The expected value}
\begin{lem}\label{lem:expectation} 
Let $M=N^{1+\varepsilon}$, $K=N^\varepsilon$ and assume
\eqref{cond expec} holds. 
Then the expected value of $R_2(f,N)(\alpha)$ is 
\[
\ave{R_2(f,N)} =\intinf f(x)dx +O(N^{-\delta}) 
\]
\end{lem}
\begin{proof}
Let $f\in C_c^\infty(\R)$. 
By using Poisson summation, we have the expansion  
 \[
 F_N(x) = \sum_{j\in \Z} f\big(N(x+j)\big) = \frac 1N\sum_{n\in \Z} \widehat f\Big(\frac nN\Big) e(nx)
 \]
 with $e(z):=e^{2\pi i z}$, which gives 
\begin{equation}\label{eq: Fourier expansion of R}
 R_2(\alpha) = \frac 1{N^2}\sum_{n\in \Z} 
 \widehat f\Big(\frac nN\Big) S_{n,N}(\alpha) 
\end{equation} 
where 
$$
S_{n,N}(\alpha) = \sum_{1\leq x \neq y \leq N}
e(\alpha n(a(x)-a(y))).
$$
Therefore the expected value is 
\[
\ave{R_2}=\intinf R_2(\alpha)\rho(\alpha)d\alpha = 
\frac 1{N^2}\sum_{n\in \Z} 
 \widehat f\Big(\frac nN\Big) \sum_{1\leq x \neq y \leq N} 
 \^\rho(n(a(x)-a(y))).
\]
The zero mode $n=0$ gives a contribution of 
\[
\frac 1{N^2}\^f(0) N(N-1) = \intinf f(x)dx (1+O(1/N)). 
\]

We split the sum over non-zero modes into two terms: Those with $1\leq |n|\leq M=N^{1+\varepsilon}$, and those with $|n|>M$. 
To treat the contribution of modes with $|n|>M=N^{1+\varepsilon}$, we use $|\^f(x)| \ll x^{-A}$  and $|\^\rho|\ll 1$ to  bound that term by 
\[
 \frac 1{N^2} \sum_{|n|>M} \Big(\frac nN \Big)^{-A} \sum_{1\leq x\neq y\leq N} 1 = \frac {N^A}{M^{A-1}} \ll \frac{1}{N^{1-\varepsilon}}
\]
on choosing $A=2/\varepsilon$. 

To bound the contribution of modes with $1\leq |n|\leq M$, we separate into a contribution of terms with $|n(a(x)-a(y))|<K$ and the rest. 

We use $|\^\rho|, |\^f|\ll 1$ to obtain that the contribution of terms with $|n(a(x)-a(y))|<K$ is 
\[
\ll \frac 1{N^2}\#\{1\leq n<N^{1+\varepsilon}, 1\leq y\neq x\leq N: n(a(x)-a(y))<N^\varepsilon\}.
\]
By \eqref{cond expec}, this is $\ll N^{-\delta}$. 

The contribution of  terms with $|n(a(x)-a(y))|>K$  is bounded using 
$$|\^\rho(n(a(x)-a(y)))\ll |n(a(x)-a(y))|^{-A}\leq K^{-A} =N^{-2}$$
and $|\^f|\ll 1$ by 
\[
\frac 1{N^2}\sum_{\substack{1\leq |n|\leq M\\1\leq x\neq y\leq N\\  |n(a(x)-a(y))|>K}} |\^f(\frac nN)|  
\^\rho( n(a(x)-a(y)) |
\ll \frac 1{N^2}\sum_{\substack{1\leq |n|\leq M\\1\leq x\neq y\leq N}} \frac 1{N^2} \leq \frac{M}{N^2}
\]
which is $\ll N^{-1+\varepsilon}$. 
\end{proof}

\subsection{The variance}
\begin{prop} \label{prop:lacunary variance} 
Assume that $a(x)$ is a sequence of   real numbers such that \eqref{cond expec} and \eqref{lacunary ineq} hold. Then 
\[
\ave{\Big| R_2(f,N)-\intinf f(x)dx \Big|^2}\ll N^{-\delta}.
\]
 \end{prop}
\begin{proof}
By Cauchy-Schwarz, 
\begin{align*}
\ave{\Big| R_2(f,N)-\intinf f(x)dx \Big|^2} & \leq 
 2\ave{\Big| R_2(f,N)-\ave{R_2} \Big|^2}
\\
& +2 \ave{\Big| \ave{R_2}-\intinf f(x)dx \Big|^2}.
\end{align*}
By Lemma~\ref{lem:expectation}, 
$$ \ave{\Big| \ave{R_2}-\intinf f(x)dx \Big|^2}\ll N^{-2\delta}.
$$ 
We now show that 
\begin{equation}\label{bound on var}
 \Var R_2=\ave{\Big| R_2(f,N)-\ave{R_2}\Big|^2} \ll N^{-\delta} 
 \end{equation}
which will  prove Proposition~\ref{prop:lacunary variance}.

 To prove \eqref{bound on var}, it suffices to show by \eqref{lacunary ineq}  that
 \begin{equation} \label{lem: variance bounded in terms of counting}
\Var R_2 \ll_f \frac{\#\mathcal S(N) }{N^4}
\end{equation}

By using the expansion \eqref{eq: Fourier expansion of R}, the variance can be written as
  \begin{equation}\label{eq: variance in terms of Fourier coefficients}
   \Var(R_{2})=\frac{1}{N^{4}} \sum_{(n_1,n_2)  \in \mathbb{Z}^{2}\setminus\{0\}} 
  \widehat{f}\Big(\frac{n_1}{N}\Big) \widehat{f}\Big(\frac{n_2}{N}\Big)
  w(\no,\nt,N).
  \end{equation}
where  for integers $n_1,n_2$, we let
$$w(\no,\nt, N) = 
\sum_{\substack{1 \leq x_1 \neq x_3 \leq N, \\ 
1\leq x_2 \neq x_4 \leq N}}
\widehat{\rho}\Big(\no (a(x_3)-a(x_1))
-\nt(a(x_4)-a(x_2))\Big)$$
and $\rho$  as in \eqref{eq: smooth ave}.

Due to the rapid decay  
of $\widehat{f}$, the contribution from the range in which $|\no|$ or $|\nt|$ exceeds $M=N^{1+ \varepsilon}$ 
is negligible, as we will argue now.
We detail only the case  $\max\{|\no|,|\nt|\} =\no \geq M$,
since the other case can be done similarly. 
We observe the trivial bound 
$\vert w(\no,\nt,N) \vert \ll N^4$. Moreover, 
$$n_1=n^{\varepsilon/2}_{1}n^{1-\varepsilon/2}_{1}
\geq n^{\varepsilon/2}_{1} N^{1+\varepsilon/2-\varepsilon^{2}/2}
$$
which, since $\varepsilon$ is small, yields $n_1 > n^{\varepsilon/2}_{1} N^{1+\varepsilon/3}$.
Hence, the contribution to the right hand side of \eqref{eq: variance in terms of Fourier coefficients} 
arising from the terms with $\max\{|\no|,|\nt|\} =\no \geq M=N^{1+ \varepsilon}$ and $n_2 \neq 0$ is
 \begin{align*}
\ll \frac 1{N^4} \sum_{\substack{\no,\nt
 \in \mathbb{Z}\setminus\{0\}\\ 
 \vert \no \vert >N^{1+ \varepsilon}}} 
 \Big(\frac{\no}{N}\Big)^{-  18/\varepsilon}  
 \sum_{n_2\neq 0} \widehat f\Big(\frac{n_2}{N}\Big)
 N^{4}
 \ll \frac 1{N^4}.
 \end{align*}
Moreover, the terms satisfying $\max\{|\no|,|\nt|\} =\no \geq N^{1+ \varepsilon}$, and $n_{2}=0$ 
are in absolute value 
 $$
 \ll \frac 1{N^4}\sum_{\vert \no \vert \geq N^{1+\varepsilon}} 
 \widehat{f}\Big(\frac{\no}{N}\Big) N^{4}\ll \frac 1{N^4}.
 $$
So, the upshot is that on the right hand side of \eqref{eq: variance in terms of Fourier coefficients} 
the sum over all $(n_1,n_2)$ with $\max(|n_1|,|n_2|)>N^{1+\varepsilon}$
   
contributes $\ll N^{-4}.$
By the rapid decay  
of $\widehat{\rho}$, we can dispose of the regime where 
 $ \vert n_1 (a(x_3)-a(x_1))-n_2(a(x_4)-a(x_2))  \vert \geq N^{\varepsilon}.
 $

By bounding $\hat{\rho}$ trivially, we find that
\[
\Var R_2 \ll \frac{\#\mathcal S(N)}{N^4} + O\Big(\frac 1{N^4}\Big) .
\]
Since $\#\mathcal S(N)\geq N^3$, we obtain \eqref{lem: variance bounded in terms of counting}. 
\end{proof}

    \section{Almost everywhere convergence: Proof of Theorem~\ref{prop:Reduce pc to ineq}}

 We now deduce  almost everywhere convergence   
 from a polynomial variance bound.

 \subsection{Preparations}\label{secion: preparations}
 We need a general property of the pair correlation function. Recall that for any sequence of points 
 $\{\theta_n\}_{n=1}^\infty \subset \R/\Z$, we defined
 \[
 R_2(f,N) = \frac 1N \sum_{1\leq j\neq k\leq N} F_N(\theta_j-\theta_k)
 \]
 with $ F_N(x) = \sum_{j\in \Z} f(N(x-j))$. 
 
 \begin{lem}\label{lem: conv subsequences}
 Suppose there is a strictly increasing sequence 
 $\{ N_m \}_{m=1}^\infty \subseteq \mathbb{Z}_{\geq 1} $, with 
 \[
\lim_{m\to \infty} \frac{N_{m+1}}{N_m} = 1
 \] 
 so that for all $f\in C_c^\infty(\R)$, 
 \begin{equation}\label{convergence on subsequence}
 \lim_{m\to \infty} R_2(f,N_m) = \intinf f(x)dx .
 \end{equation}
 Then we can pass from the sub-sequence to the set of all integers:
\begin{equation}\label{convergence on full sequence}
 \lim_{N\to \infty} R_2(f,N) = \intinf f(x)dx
 \end{equation}
 for all $f\in  C_c^\infty(\R)$.
 \end{lem}
 \begin{proof}
We will first deduce that \eqref{convergence on subsequence} holds for the indicator functions 
 \[
 I_s(x) = \begin{cases} 1,& |x|<s/2,\\ 0,&{\rm otherwise}, \end{cases}
 \]
 by approximating with smooth functions, and show that \eqref{convergence on full sequence} holds for the functions $I_s$, and then deduce by approximating a general even smooth $f\in C_c^\infty(\R)$  by linear combinations of $I_s$ 
 that \eqref{convergence on full sequence} holds for all such $f$.
 Note that for odd smooth $f\in C_c^\infty(\R)$, we have 
 $F_N(-x)=-F_N(x)$ which entails $R_2(f,N)=0$, 
 so the pair correlation function $R_2(f,N)$ 
 converges trivially to the right limit.

 From the definition of $R_2(I_s,N)$ we have a monotonicity property: 
Let $0<\varepsilon<1$.   If $  (1-\varepsilon)N'<N<N'$   and $N''<N<(1+\varepsilon)N''$ then 
 \begin{equation}\label{monotonicity prop}
  (1-\varepsilon) R_2(I_{(1-\varepsilon)s} ,N'') \leq R_2(I_s,N) \leq \frac 1{1-\varepsilon} R_2(\ I_{s/(1-\varepsilon)},N') .
 \end{equation}
    Indeed, using positivity of $ I_s$ (hence of $F_N$)
 \[
 N\cdot R_2(I_s,N) = \sum_{1\leq j\neq k\leq N} F_N(\theta_j-\theta_k) \leq \sum_{1\leq j\neq k\leq N'} F_N(\theta_j-\theta_k)  .
 \]
 Now if   $1>N/N'\geq 1-\varepsilon>0$ then since $I_s$ is even and decreasing on $[0,\infty)$,  we have 
  \[
 I_s(Ny) = I_s\Big(N' y\cdot \frac{N}{N'}\Big) 
 \leq I_s(N'y(1-\varepsilon)) = I_{s/(1-\varepsilon)}(N'y   ).
 \]
So 
 \[ 
 F_N(x) =    \sum_{j\in \Z} I_s(N \cdot(x-j)) 
 \leq \sum_{j\in \Z} I_{s/(1-\varepsilon)} (N'\cdot(x-j)) 
 = \tilde F_{N'}(x) 
\] 
where $\tilde F_{N'}(y) = \sum_{j\in \Z}  I_{s/(1-\varepsilon)}({N'}(y-j))$. Hence 
 \[
 R_2(I_s,N) \leq \frac{N'}{N} R_2(I_{s/(1-\varepsilon)},N') \leq \frac 1{1-\varepsilon}  R_2(I_{s/(1-\varepsilon) },N')
\]
which proves the upper bound in \eqref{monotonicity prop}. 
The lower bound of \eqref{monotonicity prop} follows from switching the roles of $N$ and $N''$ and inserting in the upper bound.

 Next, fix $\varepsilon\in (0,1)$ small, let $N\gg 1$, and   take $m\gg 1$ so that 
 \[
 N_m<N_{m+1}<(1+\varepsilon)N_m
 \]
 and so if $N_m\leq N<N_{m+1}$ then 
 \[
  (1-\varepsilon)N_{m+1} <N < N_{m+1}, \quad N_{m}\leq N<(1+\varepsilon)N_m .
 \]
 Then for all $s>0$, we have
 \[
 (1-\varepsilon) R_2(I_{(1-\varepsilon)s}, N_m) \leq R_2(I_s,N) \leq \frac 1{1-\varepsilon} R_2(I_{s/(1-\varepsilon)},N_{m+1}) .
 \] 
 Taking $m\to \infty$, we find by \eqref{convergence on subsequence}
 \[
 \limsup_{N\to \infty} R_2(I_s,N) \leq \frac 1{1-\varepsilon}  \intinf I_{s/(1-\varepsilon)} dx = \frac{s}{(1-\varepsilon)^2}
  \]
 and 
 \[
 \liminf_{N\to \infty} R_2(I_s,N)  \geq (1-\varepsilon) \intinf I_{ (1-\varepsilon)s} dx  = (1-\varepsilon)^2s .
 \]
 Since $\varepsilon>0$ is arbitrary, we finally obtain 
 \[
 \lim_{N\to \infty} R_2(I_s,N) = s = \intinf I_s(x)dx
 \]
 so that \eqref{convergence on full sequence} holds for all indicator functions $I_s$. 
 Therefore \eqref{convergence on full sequence} holds for all test functions $f\in C_c^\infty(\R)$.
 \end{proof} 
 
 \subsection{Proof of Theorem~\ref{prop:Reduce pc to ineq}}
 
 It suffices to show that for almost every $\alpha$ 
 in a fixed compact interval $I$ we have 
  \begin{equation}\label{full conv}
\lim_{N\to \infty}R_2(f,N)(\alpha)=\intinf f(x)dx
 \end{equation} 
 for all $f\in C_c^\infty(\R)$. 
 
Let $\rho\in C^\infty_c(\R)$ be a non-negative function majorizing the indicator function of the interval $I$: 
$\mathbf 1_{I}\leq \rho$. Then from the variance bound of Proposition~\ref{prop:lacunary variance}, we find that  
 for some $ \delta>0$, for all $f\in C_c^\infty(\R)$, 
 \[
   \int_{I} \Big |R_2(f,N_m)(\alpha)-\intinf f(x)dx \Big|^2 \rho(\alpha)d\alpha \ll_f N^{-\delta} .
 \]
 Hence for the sequence
 \[
 N_m = \lfloor m^{2/\delta} \rfloor
 \]
 we have that for almost all $\alpha \in I$,
 \begin{equation}\label{subsequential conv}
\lim_{m\to \infty} R_2(f,N_m)(\alpha) = \intinf f(x)dx
 \end{equation}
 for all $f$. Indeed, for each fixed $f$  set 
 \[
 X_m (\alpha) = \Big|R_2(f,N_m)(\alpha)-\intinf f(x)dx \Big|^2 .
 \]
Then
 \[
 \int_{I} X_m(\alpha)d\alpha \leq  \int_{-\infty}^{\infty}   X_m (\alpha)  \rho(\alpha)d\alpha  \ll   \frac 1{N_m^{\delta}} \ll \frac 1{m^2} .
 \] 
 Therefore 
  \[
 \int_{I} \Big(  \sum_{m\geq 1} X_m(\alpha) \Big) d\alpha  \leq \sum_{m\geq 1} \int_{-\infty}^\infty X_m(\alpha) d\alpha \ll_f   \sum_{m\geq 1} \frac 1{m^2}<\infty
 \]
 so that $\sum_{m\geq 1} X_m(\alpha)$ converges 
 for almost all $\alpha\in I$. Thus 
 $$\lim_{m\to \infty}X_m(\alpha)=0$$ 
 for almost all $\alpha$, i.e. \eqref{subsequential conv} holds for our specific $f$ for almost all $\alpha\in I$. 
 
 By a diagonalization argument (see \cite{RudnickSarnak})  there is a set of $\alpha$ whose complement has measure zero so that \eqref{subsequential conv} holds for all $f$. Since $N_{m+1}/N_m\to 1$, 
 we can use Lemma~\ref{lem: conv subsequences} to deduce \eqref{full conv} holds, proving Theorem~\ref{prop:Reduce pc to ineq}.

 \section{Lacunary sequences}\label{sec:lacunary}
 
 From now on, we assume that  $\{a(x)\}_{x=1}^\infty$ 
 is a lacunary sequence of (strictly) positive reals, 
 that is there is some $C>1 $ 
so that for all integers $x\geq 1$, 
\[
a(x+1)\geq C a(x)
\] 
for all $x\geq 1$. Consequently, we have for all $x\geq y\geq 1$ that
\[
a(x)\geq C^{x-y} a(y) .
\]
We will show that \eqref{cond expec} and \eqref{lacunary ineq} hold, hence proving that the pair correlation function of 
$\{\alpha a(x) \mod 1\}_{x=1}^\infty$ is Poissonian for almost all $\alpha$, that is Theorem~\ref{thm:lacunary}.

 \subsection{The condition  \eqref{cond expec}}
 \begin{lem}
Assume that the sequence $\{a(x)\}_{x=1}^\infty$ is lacunary. Then \eqref{cond expec} holds, in fact with a bound of 
\[
\#\{1\leq n\leq M, 1\leq x\neq y\leq N: n|a(x)-a(y)|<K \} 
\ll N^{2\varepsilon}.
\]
\end{lem}
\begin{proof}
Since the sequence is lacunary, we have for $y<x$ that 
\[
a(x)-a(y) \geq a(x)\Big(1-\frac 1{C^{x-y}}\Big)\geq C^x\Big(1-\frac 1{C^{x-y}}\Big)
\]
which is $\gg N^\varepsilon$ as soon as  $x\geq \epsilon \log_C N$
where $\log_C N=(\log N)/\log C$ . Hence to satisfy the inequality we must have $n<N^\varepsilon$, and $y<x\ll \log N$, so that we have at most $N^{2\varepsilon}$ solutions. 

\end{proof}

 \subsection{The condition \eqref{lacunary ineq}}
 
\begin{prop}\label{prop:counting}
Assume that  $\{a(x)\}_{x=1}^\infty$ is a lacunary sequence of positive real numbers, and that $N^\gamma\ll M\ll N^\Gamma$ for some $0<\gamma<\Gamma<2$. Then 
\[
\#\mathcal S(N)   \ll MN^2(\log M)^2.
\]
\end{prop}
In view of Theorem~\ref{prop:Reduce pc to ineq}, we deduce Theorem~\ref{thm:lacunary}. 

\begin{proof}
The proof is a  modification of  \cite[Proposition 2]{RZ pc}: We are given the inequality
\begin{equation}\label{lacunary ineq2}
| n_1 (a(x_1)-a(y_1))- n_2(a(x_2)-a(y_2)) | < K.
\end{equation}
We may assume that $n_i>0$, and $1\leq y_i<x_i\leq N$, $i=1,2$, and that $x_1\geq x_2$. In particular, we may then assume that $x_1\geq 4\log_C M \gg \log N$, 
because the number of such tuples with $x_1\ll \log N$ is at most 
$O(M^2(\log N)^4)$, which  is admissible 
(that is, $o(MN^2(\log N)^2)$) if $M=O(N^\Gamma)$ for $\Gamma<2$.

We fix $n_1,x_1,y_1$, and first show that (recall $x_1\geq x_2$) 
\begin{equation}\label{diff between s_i}
x_1 - x_2 \leq 2\log_C M .
\end{equation}
Indeed,  we have a lower bound 
\[
  n_1(a(x_1)-a(y_1)) \geq 1\cdot (a(x_1)-a(x_1-1)) 
  \geq \Big(1-\frac 1C \Big) a(x_1)
\]
(since $y\leq x_1-1$), and an upper bound 
\[
n_2(a(x_2)-a(y_2))   \leq M a(x_2) =a(x_1)M \frac{a(x_2)}{a(x_2+(x_1-x_2))} \leq \frac{M}{C^{x_1-x_2}}a(x_1) 
\]
since $a(x+h)\geq C^h a(x)$ for $h\geq 1$. Hence
\[
n_1(a(x_1)-a(y_1))
- n_2(a(x_2)-a(y_2)) 
\geq \Big(1-\frac 1C\Big) a(x_1) - \frac{M}{C^{x_1-x_2}}a(x_1) 
\]
Assuming that
$ x_1-x_2>2\log_C M $
gives in particular 
\[
1-\frac 1C -\frac{M}{C^{x_1-x_2}} >1-\frac 1C-\frac 1M 
> \frac 12 \Big(1-\frac 1C\Big)>0
\]
for sufficiently large $N$.
The condition \eqref{lacunary ineq2}   now forces
\[
\frac 12(1-\frac 1C)<  (1-\frac 1C)   - \frac{M}{C^{x_1-x_2}}  \leq \frac K{a(x_1)} \ll \frac K{C^{x_1}}
\]
which forces $x_1 \ll \log_C K\leq \varepsilon \log M$, which we assumed was not the case. Thus we may assume that $x_1-x_2>2\log_C M$, which forces 
$x_2\geq 2\log_CM$ 
since $x_1>4\log_C M$. 

Now fix $x_2$ as well; 
then $n_2$ will be determined by $y_2$, because 
\[
n_2 = \frac{n_1 (a(x_1)-a(y_1))}{a(x_2)-a(y_2)} + O\Big(\frac K{ a(x_2)-a(y_2)}\Big)
\]
and since $a(y)$ is lacunary, $K/(a(x_2)-a(y_2)) =o(1)$ 
if $ y_2 \geq \log_C N$, because 
\[
a(x_2)-a(y_2)  \geq a(x_2)-a(x_2-1)\geq a(x_2)(1-\frac 1C) \gg C^{x_2}\geq M^2
\]
since $x_2>2\log_CM$. 

So we will be done if we show that there is at most one choice of $y_2$ such that  $x_2-y_2>2\log_C M$. Indeed, if there are two pairs $(y_2,n_2)$ and  $(y_2',n_2')$ 
for which \eqref{lacunary ineq2} holds (recall all other variables are now fixed), 
with $x_2-y_2>2\log_C M$, $x_2-y_2'>2\log_C M$, then since
\[
a(y_2) \leq \frac{a(x_2)}{C^{x_2-y_2}}\leq \frac{a(x_2)}{M^2}
\]
we find that   \eqref{lacunary ineq2} implies 
\[
\begin{split}
n_1(a(x_1)-a(y_1))  &= n_2(a(x_2)-a(y_2)) +O(K) 
\\
&= n_2 a(x_2)\Big( 1+ \frac{a(y_2)}{a(x_2)} 
+O\Big(\frac K{n_2a(x_2)}\Big)\Big) 
\\
&= n_2 a(x_2)\Big( 1+ O\Big(\frac{K}{M^2}\Big)\Big)
\end{split}
\]
since $n_2a(x_2)\geq a(x_2)\geq C^{x_2} \gg M^2$ if $x_2\geq 2\log_C M$, and $a(x_2)/a(y_2)\geq C^{x_2-y_2}\gg M^2$.  
If $n_2',y_2'$ is another such pair then we also find
\[
n_1(a(x_1)-a(y_1)) = n'_2 a(x_2)
\Big( 1+ O \Big(\frac{K}{M^2} \Big)\Big)
\]
so that 
\[
n_2 a(x_2)\Big( 1+ O\Big( \frac{K}{M^2} \Big)\Big)
= n'_2 a(x_2)\Big( 1+ O\Big( \frac{K}{M^2}\Big)\Big)
\]
which gives
\[
n_2'=n_2 \Big( 1+ O\Big(\frac{K}{M^2}\Big)\Big) 
= n_2 + O\Big( \frac{K}{M}\Big)
\]
since $n_2\leq M$. Thus for $M\gg N^\gamma$, 
while $K\ll N^\varepsilon = o(M)$, we obtain $n_2'=n_2$.
\end{proof}

\end{document}